\documentclass[12pt,reqno]{amsart}
\usepackage{cases}
\usepackage{amscd}
\usepackage{amsfonts}
\usepackage{amssymb}
\usepackage{amsmath}
\usepackage{graphicx}
\usepackage{epstopdf}
\usepackage{subfigure}
\usepackage{amsthm}
\usepackage{stmaryrd}

\theoremstyle{remark}
\newtheorem{example}{\textbf{Example}}[section]
\numberwithin{equation}{section}

\usepackage{color}
\usepackage{datetime}

\makeatletter
\newcommand\figcaption{\def\@captype{figure}\caption}
\newcommand\tabcaption{\def\@captype{table}\caption}
\makeatother

\usepackage{geometry}
\usepackage{algorithm}
\usepackage{multirow}

\def\bq{\begin{equation}}
\def\eq{\end{equation}}
\def\bqs{\begin{equation*}}
\def\eqs{\end{equation*}}
\def\bsqs{\begin{subequations}}
\def\esqs{\end{subequations}}
\def\ba{\begin{aligned}}
\def\ea{\end{aligned}}
\def\br{\begin{eqnarray}}
\def\er{\end{eqnarray}}
\def\brr{\bq\begin{array}{rlll}}
\def\err{\end{array}\eq}

\def\text#1{\hbox{#1}}
\newtheorem{thm}{Theorem}[section]
\newtheorem{lem}{Lemma}[section]

\newtheorem{rem}{Remark}[section]

\newtheorem{definition}{Definition}[section]
\newtheorem{scheme}{Scheme}[section]

\newcommand{\bsub}{\begin{subequations}}
\newcommand{\esub}{\end{subequations}$\!$}

\title[LEQRK-DG-PC]{Energy stable Runge-Kutta discontinuous Galerkin schemes for fourth order gradient flows}

\author[H.~Liu, P. ~Yin]{Hailiang Liu$^\dagger$ and Peimeng Yin$^\S$}
\address{$^\ddagger$ Iowa State University, Department of Mathematics, Ames, IA 50011} \email{hliu@iastate.edu}
\address{$^\S$ Wayne State University, Department of Mathematics, Detroit, MI 48202} \email{pyin@wayne.edu}

\begin{document}

\begin{abstract} We present unconditionally energy stable Runge-Kutta (RK) discontinuous Galerkin (DG) schemes for solving a class of fourth order gradient flows. Our algorithm is geared toward arbitrarily high order approximations in both space and time, while energy dissipation remains preserved without imposing any restriction on time steps and meshes. We achieve this in two steps. First, taking advantage of the penalty free DG method introduced by Liu and Yin [J  Sci. Comput. 77:467--501, 2018] for spatial discretization, we reformulate an extended linearized ODE system by the energy quadratization (EQ) approach. Second, we apply an s-stage algebraically stable RK method for temporal discretization. The resulting fully discrete DG schemes are linear and unconditionally energy stable. In addition, we introduce a prediction-correction procedure to improve both the accuracy and stability of the scheme. We illustrate the effectiveness of the proposed schemes by numerical tests with benchmark problems.
\end{abstract}

\maketitle

\bigskip



\section{Introduction}
In this paper, we are concerned with arbitrarily high order numerical approximations to a class of fourth order gradient follows,
\begin{align}\label{fourthPDE}
u_t  = - \mathcal{L}^2 u  -\Phi'(u), \; x\in \Omega, \; t>0,
\end{align}
where $\mathcal{L}=-\left(\Delta + a\right)$ is a second-order operator with a physical parameter $a$ and $\Phi$ is a nonlinear function bounded from below. The model equation (\ref{fourthPDE}) governs the evolution of a scalar time-dependent unknown $u=u(x,t)$ in a convex  domain $\Omega \subset \mathbb{R}^d$ and it describes important physical processes in nature. Typical examples  of (\ref{fourthPDE}) include the Swift--Hohenberg equation \cite{SH77} and the extended Fisher-Kolmogorov equation \cite{DS88, PT95}. There are other gradient flows such as the classical Allen--Cahn equation \cite{AC79} and Cahn--Hilliard  equation \cite{CH58}.  The common feature of gradient flow models is that the dynamics is driven by minimizing free energy.

We consider boundary conditions of form
\begin{equation}\label{BC}
\text{(i)} \ u \text{ is periodic};  \quad \text{or (ii)}   \ \partial_\mathbf{n} u=\partial_\mathbf{n} \Delta u=0, \quad x\in \partial \Omega,
\end{equation}
where $\mathbf{n}$ stands for the unit outward normal to the boundary $\partial \Omega$. With such boundary conditions, equation  (\ref{fourthPDE}) indeed features the energy dissipation property:
\bq\label{engdisO}
\frac{d}{dt}\mathcal{E}(u) =  -\int_{\Omega}|u_t|^2dx \leq 0,
\eq
where the free energy
\bq\label{eqpdf}
\mathcal{E}(u) = \int_\Omega \frac{1}{2}\left(\mathcal{L} u \right)^2 + \Phi(u)dx.
\eq
The model equation is nonlinear, its analytical solution is intractable. Hence designing accurate, efficient, and energy stable algorithms to solve it  becomes essential. This energy dissipation law as a fundamental property of (\ref{fourthPDE}) has been explored in high order numerical approximations \cite{FK99, LY19, LY20SAV}. It was shown to be crucial to eliminate numerical results that are not physical. In this paper, we construct unconditionally energy stable and arbitrarily high order schemes to solve the above model problem, for which we use discontinuous Galerkin (DG) methods for spatial discretization, and high order Runge-Kutta (RK) methods for time discretization.

\subsection{Related work} In the literature, there has been rapid development of different methods for simulating gradient flow models including (\ref{fourthPDE}), see e.g., \cite{CP02, CPWV97, GN12, Le17, SEVDPC17,Ey98, WZZ14, XT06, SY10, Y16, ZYGZLY18, SY18, CWWW19}.
They vary either in the spatial discretization or the time discretization, while the latter typically emphasizes preserving the energy dissipation property with no or mild time step restrictions. Let us briefly discuss existing works closely related to what we do here.

{\bf DG spatial discretization.} It is known that for equations containing higher order spatial derivatives, DG discretization entails subtle  difficulties in defining numerical fluxes. Several approaches have been developed to deal with the difficulties, including the local DG (LDG) methods \cite{YS02, DS09, WZS17}, the mixed symmetric interior penalty (SIPG) methods \cite{FK07, FLX16, WKG06, FKU17}, 
and the ultra-weak DG \cite{CS08}. To avoid certain drawbacks of these methods, a penalty free DG method was introduced in \cite{LY18}, where the symmetric structure of the model (\ref{fourthPDE}) is essentially used. This method still inherits the advantages of the usual DG methods, such as higher accuracy, flexibility in hp-adaption, capacity to handle domains with complex geometry \cite{KWLK00, HW07, Ri08, Shu09}, its distinct feature lies in numerical fluxes without using any interior penalty. This is the spatial discretization we shall follow in this work.

{\bf EQ reformulation and time discretization.}
To keep the energy stability for gradient flow models, several time discretization techniques are  available in the literature, including the so-called convex splitting \cite{Ey98, WZZ14}, and the stability approach \cite{XT06, SY10}. The former leads to  nonlinear schemes, and the later often imposes restrictions on nonlinear terms in the model. The energy quadratization (EQ) approach introduced in \cite{Y16, ZYGZLY18} turned to be rather general that it can be applied to a class of gradient flow models if the free energy is bounded below.
Based on the idea of EQ, the scalar auxiliary variable (SAV) approach was introduced later in \cite{SY18}, where linear systems only with constant coefficients need to be solved. Several extensions of EQ and SAV have been further explored in \cite{CZY18, HAX19, LL19, XYZX19}. Earlier EQ based schemes are mostly up to 2nd order accurate in time, until recent works \cite{GZW20SAV, GZW20SIAM}, where the EQ formulation is combined with the Runge-Kutta methods to achieve high order in time schemes. Note that their schemes are fully nonlinear so that the solution existence and uniqueness are not guaranteed for large time steps. This issue is further addressed in \cite{GZW20} in which the obtained schemes are unconditionally energy stable and linear.  We note that many existing EQ based schemes such as \cite{Y16, ZYGZLY18, GZW20, GZW20SAV, GZW20SIAM} use mainly finite-difference or spectral methods for spatial discretization. New difficulties arise when coupling EQ with the DG discretization, as shown in \cite{LY19, LY20}.

{\bf Integration of DG with EQ.} Integration of EQ formulation with DG for solving (\ref{fourthPDE}) began with \cite{LY19}, where up to 2nd order (in time) IEQ-DG schemes are introduced. These schemes are shown to be unconditionally energy stable independent of the size of time steps, and easy to implement without resorting to any iteration method. A key point for the success in the scheme formulation is that the auxiliary energy variable is updated in point-wise manner, and then projected back into the DG space. This strategy of constructing the IEQ-DG schemes was further extended to solve the Cahn--Hilliard equation \cite{LY20}, where the spatial discretization is based on the DDG method \cite{LY09, LY10}. We note that a direct integration of the DG method with the SAV approach \cite{SY18} will lead to linear systems involving dense coefficient matrices, hence rather expensive to solve. A special procedure was introduced in \cite{LY20SAV} as a rescue, so the resulting SAV-DG schemes become well positioned to solve (\ref{fourthPDE}). However, all these EQ/SAV based DG schemes are no more than second order in time.


{\bf Present investigation.} The main purpose of this paper is to construct unconditionally energy stable DG schemes coupled with arbitrarily high order time discretization. We achieve this in two steps. First, taking advantage of the penalty free DG method introduced in \cite{LY19}, we reformulate an extended linearized ODE system by the energy quadratization (EQ) approach. Second, following \cite{GZW20}, we apply an s-stage algebraically stable RK method for temporal  discretization. The resulting fully discrete DG schemes are linear and unconditionally energy stable.

\subsection{Our contribution}
\begin{itemize}
\item We construct the linear energy quadratized Runge-Kutta DG (LEQRK-DG) schemes to solve (\ref{fourthPDE}).
In this construction we use a higher order interpolation in linearizing the extended EQ system, and a spatial projection for updating the auxiliary variable back into the DG space.
\item We show that the LEQRK-DG schemes feature a discrete energy dissipation law for any time steps, hence these schemes are called  unconditionally energy stable. We also propose LEQRK-DG-PC schemes following a prediction-correction procedure to improve both accuracy and stability of the LEQRK-DG schemes.
\item We conduct experiments on benchmark examples to evaluate the performance of LEQRK-DG-PC method. First, we present numerical results to show the high order of spatial and temporal accuracy of the proposed schemes, and the energy dissipating properties of numerical solutions. Second, we conduct experiments on some two-dimensional pattern formation problems, all of which demonstrate the good performance of the LEQRK-DG-PC method.
\end{itemize}

\subsection{Organization}
In Section 2, we formulate a unified semi-discrete DG method for the gradient flow (\ref{fourthPDE}) subject to two different boundary conditions. In Section 3, we present LEQRK-DG schemes, and show the energy dissipation law. We also present the LEQRK-DG-PC method following a prediction-correction procedure. In Section 4, we verify the good performance of LEQRK-DG-PC method using several benchmark numerical examples. Finally some concluding remarks are given in Section 5.

\section{Spatial DG discretization }
We derive mathematical formulation for our method. We begin with rewriting (\ref{fourthPDE}) as a mixed form
\begin{equation}\label{mix}
\left \{
\begin{array}{rl}
    u_t = &- \mathcal{L} q-\Phi'(u),\\
    q = & \mathcal{L} u.
\end{array}
\right.
\end{equation}
Such reformulation is not unique, the symmetric feature of (\ref{mix}) is essential for our DG method without the use of any interior penalty \cite{LY18}. Let us recall some conventions of the DG discretization introduced in \cite{LY18}. Let the domain $\Omega$ be a union of  shape regular meshes $\mathcal{T}_h=\{K\}$, with the mesh size $h_K = \text{diam}\{K\}$ and $h=\max_{K} h_K$.  We denote the set of the interior interfaces by $\Gamma^0$,  the set of all boundary faces by $\Gamma^\partial$,  and the discontinuous Galerkin finite element space by
$$
V_h = \{v\in L^2(\Omega) \ : \ v|_{K} \in P^k(K), \ \forall K \in \mathcal{T}_h \},
$$
where $P^k(K)$ denotes the set of polynomials of degree no more than $k$ on element $K$. If the normal vector on the element interface $e\in \partial K_1 \cap \partial K_2$ is oriented from $K_1$ to $K_2$, then the average $\{\cdot\}$ and the jump $[\cdot]$ operator are defined by
$$
\{v\} = \frac{1}{2}(v|_{\partial K_1}+v|_{\partial K_2}), \quad [v]=v|_{\partial K_2}-v|_{\partial K_1}, 
$$
for any function $v \in V_h$,  where $v|_{\partial K_i} \ (i=1,2)$ is the trace of $v$ on $e$ evaluated from element $K_i$.

Our DG discretization based on the mixed form (\ref{mix}) is to find $(u_h(\cdot, t), q_h(\cdot, t))\in V_h\times V_h  $ such that
\begin{subequations}\label{SemiDG}
\begin{align}
& (u_{ht}, \phi) = -G(q_h, \phi) - (\Phi'(u_h), \phi), \\
& (q_h, \psi)=G(u_h, \psi),
\end{align}
\end{subequations}
for all $\phi, \ \psi \in V_h$.  The precise form of $G(\cdot, \cdot)$ depending on the types of boundary conditions is given as follows:
\bq \label{A0}
\ba
G(w,v)= &\sum_{K\in \mathcal{T}_h} \int_K \left( \nabla w \cdot \nabla v  -a w v \right)dx + \sum_{e\in \Gamma^0} \int_e \left( \{\partial_\nu w\}[v]+ [w]\{\partial_\nu v\} \right)ds\\
&+\frac{\theta}{2} \int_{\Gamma^{\partial} } \left( \{\partial_\nu w\}[v]+ [w]\{\partial_\nu v\} \right)ds,
\ea
\eq
where $\theta=1$ for (i) of (\ref{BC}) and $\theta=0$ for (ii) of (\ref{BC}).
Note that for periodic case (i) the left boundary and the right boundary are considered as same, for which we use the factor $1/2$ to avoid recounting. The initial data for $u_h$ is taken as the piecewise $L^2$ projection, denoted by  $u_h(x,0)=\Pi u_0(x)$.

The remarkable property of the above DG scheme is that
the discrete energy of form
$$
\mathcal E(u_h, q_h):= \frac{1}{2}\|q_h\|^2 + \int_{\Omega} \Phi(u_h)dx
$$
admits a discrete dissipation law \cite{LY18}:
\begin{align}\label{dd}
\frac{d}{dt}\mathcal{E}(u_h, q_h) =  -\int_{\Omega}|u_{ht}|^2dx \leq 0.
\end{align}

\section{Time discretization}
This section is devoted to arbitrarily higher order time discretization of the DG formulation (\ref{SemiDG}).

First, we choose  $C_0$ so that $\Phi(w)+C_0 > 0, \ \forall w \in \mathbb{R}$, and introduce
\begin{align}\label{hw}
 H(w)= \frac{\Phi'(w)}{\sqrt{\Phi(w)+C_0}}.
\end{align}
Then the IEQ reformulation of (\ref{SemiDG}) requires to find $(u_h(\cdot, t), q_h(\cdot, t)) \in V_h  \times V_h$ and $U$ such that
\begin{subequations}\label{SemiDG++}
\begin{align}
(u_{ht}, \phi) = &- G(\phi, q_h) -\left(H(u_h)U,\phi \right),\\
(q_h, \psi) = & G(u_h,\psi),\\
U_{t} =&  \frac{1}{2}H(u_h) u_{ht},
\end{align}
\end{subequations}
for all $\phi, \psi \in V_h $. The initial data for the above scheme is chosen as
$$
u_h(x, 0)=\Pi u_0(x), \quad U(x, 0)=\sqrt{\Phi(u_0(x))+C_0},
$$
where $\Pi$ denotes the piecewise $L^2$ projection into $V_h$. Note that $U \not \in V_h$.

There are two steps involved in the time discretization of  (\ref{SemiDG++}). First, we utilize the numerical solutions of $u_h$ for $t\leq t_n$ to obtain a high order approximation $u_h^{*}$, and replace the semi-discrete DG scheme (\ref{SemiDG++})
by
\begin{subequations}\label{LSemiDG}
\begin{align}
(u_{ht}, \phi) = &- G(\phi, q_h) -\left(H(u_h^{*})U,\phi \right),\\
(q_h, \psi) = & G(u_h,\psi),\\
U_{t} =&  \frac{1}{2}H(u_h^{*}) u_{ht}.
\end{align}
\end{subequations}
This linear scheme can be further solved in $t\in (t_n, t_{n+1}]$ by a high order ODE solver.
We should point out that the above treatment does not destroy the energy dissipation property. Since for the modified energy functional
$$
E(q_h, U)=\frac{1}{2}\|q_h\|^2 + \|U\|^2 = \mathcal E(u_h, q_h) +C_0|\Omega|
$$
still satisfies
\bqs
\frac{d}{dt}E(q_h, U) = -\int_{\Omega}|u_{ht}|^2dx \leq 0.
\eqs
Recall that for ODE of from
$
y_t =f(t, y),
$
the general s-stage Runge--Kutta (RK) method has the form
$$
y^{n+1}=y^n +\tau \sum_{i=1}^s b_i k_i,
$$
where
$$
k_i=f(t_n+c_i \tau,  y^n+\tau\sum_{j=1}^s a_{ij}k_j), \quad i=1,\cdots, s.
$$
Here for consistency the RK coefficients satisfy $c_i = \sum_{j=1}^s a_{ij}$ and  $\sum_{i=1}^s b_i=1$.
For the convenience in applying the RK method to the semi-discrete DG schemes, we introduce the operator $L_h$ by
\begin{align}\label{lh}
(L_h v, \phi)=G(v, \phi) \quad \forall \phi \in V_h.
\end{align}
\subsection{LEQRK-DG schemes}
Applying a s-stage RK method to scheme  (\ref{LSemiDG}), we obtain the following LEQRK-DG scheme.
\begin{scheme}\label{LEQRKDG} (s-stage LEQRK-DG scheme) For given $u_h^{n}, U^{n}$ and $u^{n,*}_{ih}=u^*_h(x,t_n+c_i\tau)$, we find  $(u_h^{n+1}, q_h^{n+1}, U_h^{n+1})$  by
\begin{subequations}\label{RKIEQDG2}
\begin{align}
u_h^{n+1} = & u_h^n + \tau \sum_{i=1}^s b_{i}\xi_{ih},\\
q_h^{n+1} = & L_hu_h^{n+1},\\
U^{n+1}= & U_h^n+ \tau \sum_{i=1}^s b_{i}l_i,\\
U_h^{n+1} = & \Pi U^{n+1},
\end{align}
\end{subequations}
where $\xi_{ih} \in V_h$ and $l_i$ are determined by
  \begin{subequations}\label{RKIEQDG0}
	\begin{align}
	\left( \xi_{ih}, \phi \right) = & - G(\tilde{q}_{ih},\phi)-\left(H(u^{n,*}_{ih})\tilde{U}_i,\phi \right), \quad  i=1,2,\cdots,s\\
	(\tilde{q}_{ih}, \psi) = & G(\tilde{u}_{ih},\psi), \quad \forall \phi, \psi \in V_h, \\
    l_i = & \frac{1}{2} H(u^{n,*}_{ih}) \xi_{ih},
	\end{align}
\end{subequations}
and
\begin{subequations}\label{RKIEQDG1}
\begin{align}
\tilde{u}_{ih} = &u_h^n + \tau \sum_{j=1}^s a_{ij}\xi_{jh},\\
\tilde{U}_i = & U_h^n + \tau \sum_{j=1}^s a_{ij}l_j.
\end{align}
\end{subequations}
\end{scheme}
\begin{definition}\label{rkstab} (Algebraically stable RK method \cite{BB79}) A RK method is algebraically stable if the RK coefficients satisfy stability conditions
\begin{align}\label{as}
b_i\geq 0, \quad i=1,2,\cdots,s, \quad \text{and} \quad M \; \text{is positive semi-definite},
\end{align}
where $M$ is a symmetric matrix with elements
\bq\label{Mdef}
M_{ij}=b_ia_{ij} + b_ja_{ji} - b_ib_j.
\eq
\end{definition}
Next, we show that the algebraically stable LEQRK-DG scheme is unconditionally energy stable.
\begin{thm}
The LEQRK-DG scheme with its RK coefficients satisfying the stability condition
(\ref{as}) is uniquely solvable for any $\tau>0$ and unconditionally energy stable  in the sense that
\bq\label{energystab}
E_h^{n+1} \leq E_h^n -\tau \sum_{i=1}^{s}b_i\|\xi_{ih}\|^2,
\eq
where the energy
$$
E_h^n := E(q_h^n, U_h^n) = \frac{1}{2} \|q_h^n\|^2+\|U_h^n\|^2.
$$
\end{thm}
\begin{proof} In order to prove  (\ref{energystab}), we use $\|U_h\|\leq ||U\|$ to obtain
$$
E_h^{n+1}-E_h^n \leq  \frac{1}{2} (\|q_h^{n+1}\|^2-\|q_h^n\|^2) +
(\|U^{n+1}\|^2 - \|U_h^n\|^2)
$$
and estimate two terms on the right, respectively. First we have
\bqs
\ba
\frac{1}{2}\left( \|q_h^{n+1}\|^2-\|q_h^{n}\|^2\right)=
& (q_h^{n+1}-q_h^{n},q_h^{n+1})- \frac{1}{2}\|q_h^{n+1}-q_h^{n}\|^2\\
= & G(u_h^{n+1}-u_h^{n},q_h^{n+1})- \frac{1}{2}\|q_h^{n+1}-q_h^{n}\|^2\\
= & \tau \sum_{i=1}^{s}b_i G(q_h^{n+1}, \xi_{ih}) - \frac{1}{2}\|q_h^{n+1}-q_h^{n}\|^2.\\
\ea
\eqs
Note that from (\ref{RKIEQDG2}ab), (\ref{RKIEQDG0}b) and (\ref{RKIEQDG1}a), we have
\bqs
\ba
q_h^{n+1} = & L_hu_h^{n+1} = L_hu_h^n + \tau \sum_{j=1}^s b_{j}L_h \xi_{jh}, \\
\tilde{q}_{ih} =& L_h\tilde{u}_{ih} = L_h u_h^n + \tau \sum_{j=1}^s a_{ij}L_h\xi_{jh}.
\ea
\eqs
This gives
\bqs
q_h^{n+1} = \tilde{q}_{ih} + \tau \left( \sum_{j=1}^s b_{j}L_h \xi_{jh} -\sum_{j=1}^s a_{ij}L_h\xi_{jh} \right),
\eqs
which implies
\bq\label{qhn1}
G(q_h^{n+1}, \xi_{ih}) = G(\tilde{q}_{ih}, \xi_{ih}) + \tau \left( \sum_{j=1}^s b_{j}G(L_h \xi_{jh},\xi_{ih}) -\sum_{j=1}^s a_{ij}G(L_h \xi_{jh} ,\xi_{ih})\right).
\eq
Setting $\phi=-\xi_{ih}$ in (\ref{RKIEQDG0}a), we have
\bq\label{xin1}
\ba
-\|\xi_{ih}\|^2 = & G(\tilde{q}_{ih},\xi_{ih}) + \left(H(u^{n,*}_{ih})\tilde{U}_i,\xi_{ih} \right) \\
= & G(\tilde{q}_{ih},\xi_{ih}) + 2(\tilde{U}_i, l_i),
\ea
\eq
where we have used (\ref{RKIEQDG0}c) in the last step.
Combining (\ref{xin1}) with (\ref{qhn1}) gives
\bqs
G(q_h^{n+1}, \xi_{ih}) = -\|\xi_{ih}\|^2 - 2(\tilde{U}_i, l_i) + \tau \left( \sum_{j=1}^s b_{j}G(L_h\xi_{jh},\xi_{ih}) -\sum_{j=1}^s a_{ij}G(L_h\xi_{jh} ,\xi_{ih})\right).
\eqs
Further, using  (\ref{RKIEQDG2}ab), we obtain
\bqs
\ba
\frac{1}{2}\|q_h^{n+1}-q_h^{n}\|^2 = &  \frac{1}{2}\left( L_hu_h^{n+1} - L_h u_h^{n}, L_h u_h^{n+1} - L_h u_h^{n}\right) \\
= &
\frac{1}{2}\tau^2  \sum_{i,j=1}^s b_ib_j \left( L_h \xi_{ih},  L_h \xi_{jh} \right) \\
= & \frac{1}{2}\tau^2  \sum_{i,j=1}^s b_ib_j G(\xi_{ih},  L_h \xi_{jh}).
\ea
\eqs
For the second term we use (\ref{RKIEQDG2}c) to obtain
\bqs
\ba
\|U^{n+1}\|^2-\|U_h^{n}\|^2=& 2(U^{n+1}, U^{n+1}-U_h^{n})-\|U^{n+1}-U_h^{n}\|^2 \\
= & 2(U^{n+1}, \tau \sum_{i=1}^s b_{i}l_i)-(U^{n+1}-U_h^{n},U^{n+1}-U_h^{n}) \\
= & 2\tau \sum_{i=1}^s b_{i} (U^{n+1}, l_i)-\tau^2  \sum_{i,j=1}^s b_ib_j (l_i, l_j).
\ea
\eqs
Putting together all these estimates,
\bqs
\ba
E_h^{n+1}-E_h^n 
\leq & - \frac{1}{2}\tau^2  \sum_{i,j=1}^s b_ib_j G(\xi_{ih},  L_h \xi_{jh}) -\tau^2  \sum_{i,j=1}^s b_ib_j (l_i, l_j) \\
& \qquad  + \tau \sum_{i=1}^{s}b_i G(q_h^{n+1}, \xi_{ih}) + 2\tau \sum_{i=1}^s b_{i} (U^{n+1}, l_i)\\
\leq & -\tau^2  \sum_{i,j=1}^s b_ib_j (l_i, l_j) -  \tau \sum_{i=1}^s b_{i}\|\xi_{ih}\|^2 + 2\tau \sum_{i=1}^s b_{i} (U^{n+1}-\tilde{U}_i, l_i) \\
& + \tau^2 \left( \frac{1}{2}\sum_{i,j=1}^s b_{i} b_{j}G(L_h\xi_{jh},\xi_{ih}) -\sum_{i,j=1}^s b_{i}a_{ij}G(L_h\xi_{jh},\xi_{ih}) \right).
\ea
\eqs
Subtracting (\ref{RKIEQDG1}b) from (\ref{RKIEQDG2}c) gives
\bqs
\ba
U^{n+1}-\tilde{U}_i = \tau \sum_{j=1}^s b_{j}l_j - \tau \sum_{j=1}^s a_{ij}l_j.
\ea
\eqs
Hence
\bqs
\ba
2\tau \sum_{i=1}^s b_{i} (U^{n+1}-\tilde{U}_i,l_i)  = 2\tau^2 \left( \sum_{i,j=1}^s b_{i}b_{j}(l_i, l_j) - \sum_{i,j=1}^s b_{i}a_{ij}(l_i, l_j)\right).
\ea
\eqs
Combining the results above, we have
\begin{align*}
E_h^{n+1}-E_h^n & \leq - \tau \sum_{i=1}^s b_{i}\|\xi_{ih}\|^2 - \frac{\tau^2}{2} \sum_{i,j=1}^s M_{ij} (L_h\xi_{ih},  L_h\xi_{jh}) - \tau^2 \sum_{i,j=1}^s M_{ij} (l_i, l_j) \\
& \leq - \tau \sum_{i=1}^s b_{i}\|\xi_{ih}\|^2,
\end{align*}
where we have used (\ref{as}).

It is left to prove the unique solvability of the fully discrete scheme, for which it suffices to prove the linear scheme admits only a zero solution if $u_h^n=0$ and $U^n=0$.  In fact from $E_h^n=0$, the energy dissipation inequality above tells that
$$
\frac{1}{2}\|q_h^{n+1}\|^2 + \|U^{n+1}\|^2  + \tau \sum_{i=1}^s b_{i}\|\xi_{ih}\|^2 + \frac{\tau^2}{2} \sum_{i,j=1}^s M_{ij} (L_h\xi_{ih},  L_h\xi_{jh}) + \tau^2 \sum_{i,j=1}^s M_{ij} (l_i, l_j)\leq 0.
$$
This therefore ensures that
$$
q_h^{n+1}=0, U^{n+1}=0, \quad i=1,\cdots, s,
$$
and $ b_i \xi_{ih}=0$ for $i=1\cdots s$, so $u_h^{n+1}=\tau \sum_{i=1}^s b_i\xi_{ih}=0$.
\end{proof}

\begin{rem}\label{Bchoice} To ensure the energy stability it suffices to take $C_0 >-\inf \Phi(u)$.
However, a larger $C_0$ can help to reduce the spatial projection error when associated with
the DG discretization. 
For example, let $\Pi U_0$ be the piecewise $L^2$ projection of $U_0$ in $V_h$ based on $P^1$ polynomials, then the projection error is known as
\bq\label{projerr}
\|U_0-\Pi U_0\| = Ch|U_0|_{H^1(\Omega)},
\eq
where $C$ independent of $h$ and $U_0$. Note that
$$
|U_0|^2_{H^1(\Omega)} = \sum_{K \in \Omega}\int_K\left( \frac{\Phi'(u_0)}{\sqrt{\Phi(u_0)+C_0}} \right)^2|\nabla u_0|^2dx,
$$
from which we see that a larger $C_0$ will reduce the total error.
\end{rem}
\begin{rem} System (\ref{RKIEQDG0}) may be put as a closed linear system as
    \bqs
    \begin{aligned}
      \left( \xi_{ih}, \phi \right) + \frac{\tau}{2} \sum_{j=1}^s a_{ij}\left(H(u^{n,*}_{ih})^2 \xi_{jh},\phi \right) + G(\tilde{q}_{ih},\phi)= & -\left(H(u^{n,*}_{ih})U_h^n,\phi \right),\\
	\tau \sum_{j=1}^s a_{ij}G(\xi_{jh},\psi) -(\tilde{q}_{ih}, \psi) = & -G(u_h^n,\psi),
    \end{aligned}
    \eqs
where the first equation is obtained by plugging (\ref{RKIEQDG1}b) as well as (\ref{RKIEQDG0}c) into (\ref{RKIEQDG0}a), and the second equation is obtained by plugging (\ref{RKIEQDG1}a) into (\ref{RKIEQDG0}b).
\end{rem}
\begin{rem} A variety of algebraically stable RK methods have been introduced in the literature, see, e.g.,  \cite{BB79}.  Here we present three methods in the form of the Butcher tableau. Qin and Zhang's two-stage, second order diagonally implicit RK method \cite{QZ92}
\bq\label{RK2nd}
\begin{array}{c|c}
\mathbf{c} &  A \\
\hline
 & \mathbf{b^T} \\
\end{array}
=
\begin{array}{c|cc}
\frac{1}{4} &  \frac{1}{4} & 0 \\
\frac{3}{4} &  \frac{1}{2} & \frac{1}{4} \\
\hline
 & \frac{1}{2} & \frac{1}{2} \\
\end{array},
\quad
M=\left[
\begin{array}{cc}
0 & 0 \\
0 & 0\\
\end{array}
\right],
\eq
Crouzeix's two-stage, third order diagonally implicit RK method \cite{N74},
\bq\label{RK3rd}
\begin{array}{c|c}
\mathbf{c} &  A \\
\hline
 & \mathbf{b^T} \\
\end{array}
=
\begin{array}{c|cc}
\frac{1}{2}+\frac{\sqrt{3}}{6} &  \frac{1}{2}+\frac{\sqrt{3}}{6} & 0 \\
\frac{1}{2}-\frac{\sqrt{3}}{6} &  -\frac{\sqrt{3}}{3} & \frac{1}{2}+\frac{\sqrt{3}}{6} \\
\hline
 & \frac{1}{2} & \frac{1}{2} \\
\end{array},
\quad
M=\left(\frac{1}{4}+\frac{\sqrt{3}}{6}\right)\left[
\begin{array}{cc}
1 & -1 \\
-1 & 1\\
\end{array}
\right],
\eq
and the two-stage, fourth order Gauss-Legendre method \cite{I96}:
\bq\label{RK4th}
\begin{array}{c|c}
\mathbf{c} &  A \\
\hline
 & \mathbf{b^T} \\
\end{array}
=
\begin{array}{c|cc}
\frac{1}{2}-\frac{\sqrt{3}}{6} &  \frac{1}{4} & \frac{1}{4}-\frac{\sqrt{3}}{6} \\
\frac{1}{2}+\frac{\sqrt{3}}{6} &  \frac{1}{4}+\frac{\sqrt{3}}{6} & \frac{1}{4} \\
\hline
 & \frac{1}{2} & \frac{1}{2} \\
\end{array},
\quad
M=\left[
\begin{array}{cc}
0 & 0 \\
0 & 0\\
\end{array}
\right].
\eq
These RK methods will be adopted in our numerical experiments.
\end{rem}
\begin{rem}
For RK methods with Butcher tableau
\bqs\label{OneRK2nd}
\begin{array}{c|c}
\mathbf{c} &  A \\
\hline
 & \mathbf{b^T} \\
\end{array}
=
\begin{array}{c|c}
1 &  1 \\
\hline
 & 1 \\
\end{array},
\quad
\begin{array}{c|c}
\mathbf{c} &  A \\
\hline
 & \mathbf{b^T} \\
\end{array}
=
\begin{array}{c|c}
\frac{1}{2} & \frac{1}{2} \\
\hline
 & 1 \\
\end{array},
\eqs
Scheme \ref{LEQRKDG} reduces to the first order, second order IEQ-DG schemes in \cite{LY19}, respectively.
\end{rem}
To complete Scheme \ref{LEQRKDG}, we discuss how to prepare $u_{h}^*$, hence $u_{ih}^{n, *}$. For $n=0$, we take
$$
u_h^0=\Pi u_0, \quad u_{ih}^{0, *}=u_h^0.
$$
For $n\geq 1$, we construct a Lagrangian interpolating polynomial $u_h^*$ based on $s+2$ points:
$$
(t_{n-1}, u_h^{n-1}), (t_{n-1}+c_i\tau, \tilde u_{ih}), (t_n, u^n_h),$$
and set
$$
u_{ih}^{n, *}=u_h^*(x, t_n+c_i\tau).
$$
However, two drawbacks might show up with this simple interpolation: (i) when $s$  is large, interpolating polynomials may be highly oscillatory, leading to instability or inaccuracy of the extrapolation from $[t_{n-1}, t_n]$ to $(t_n, t_{n+1}]$; (ii) the order of accuracy of the interpolation can be lower than the order of the RK method, putting another restriction on the overall accuracy of the resulting scheme. The Gauss-Legendre method in (\ref{RK4th}) is a such example.

\subsection{LEQRK-DG-PC method}
In order to improve the stability as well as the accuracy of Scheme \ref{LEQRKDG} we propose a prediction-correction method, following the strategy in \cite{GO16, GZW20}.
To do so, we also need the Lagrangian interpolation 
polynomial $U_h^*(x,t)$ based on the interpolation points
$$
(t_{n-1}, U_h^{n-1}), (t_{n-1}+c_i\tau, \tilde{U}_{ih}) \text{ and } (t_{n}, U_h^{n}), i=1,2,\cdots, s.
$$
Here, $\tilde{U}_{ih}=\Pi \tilde{U}_{i}$ is the piecewise $L^2$ projection of $\tilde{U}_{i}$ in (\ref{RKIEQDG1}b) from $(t_{n-1},t_n]$.

\begin{scheme}\label{LEQRKDGPC} (s-stage LEQRK-DG-PC scheme)
For given $u_h^n, U^n$, $u_h^*(x,t_n+c_i \tau)$ and $U_h^*(x,t_n+c_i \tau)$, $i=1,2,\cdots,s$, a s-stage LEQRK-DG-PC scheme is given as follows.
\begin{description}
  \item[Prediction]   Set $\tilde u_{ih}^{0} = u_h^*(x,t_n+c_i \tau), \tilde U_{ih}^{0} = U_h^*(x,t_n+c_i \tau)$, we iteratively
  solve 
  \begin{subequations}\label{PRKIEQDG0}
	\begin{align}
	\left(\xi_{ih}^{m+1}, \phi \right) = & - G(\tilde{q}_{ih}^{ m+1},\phi)-\left(H(\tilde u^{m}_{ih}) \tilde U_{ih}^{m},\phi \right),\\
	(\tilde{q}_{ih}^{m+1}, \psi) = & G(\tilde{u}_{ih}^{m+1},\psi),\quad \forall \phi, \psi \in V_h,
	\end{align}
\end{subequations}
and
\begin{subequations}\label{PRKIEQDG1}
\begin{align}
\tilde{u}_{ih}^{m+1} = &u_h^n + \tau \sum_{j=1}^s a_{ij}\xi_{jh}^{m+1},\\
l_i^{m+1} = & \frac{1}{2} H(\tilde u^{m+1}_{ih}) \xi_{ih}^{m+1},\\
\tilde U_i^{m+1} = & U_h^n + \tau \sum_{j=1}^s a_{ij}l_j^{m+1},\\
\tilde U_{ih}^{m+1} = & \Pi \tilde U_i^{m+1}.
\end{align}
\end{subequations}
If $\max_i \|\tilde u_{ih}^{m+1} - \tilde u_{ih}^{m}\|_\infty<Tol$, we stop the iteration and set $u^{n,*}_{ih}=\tilde u_{ih}^{m+1}$; otherwise, we set $u^{n,*}_{ih}=\tilde u_{ih}^{L}$, where  $L> 0$ is a priori given integer.
\item[Correction]
With the predicted $u^{n,*}_{ih}$, we apply Scheme \ref{LEQRKDG}
to update the numerical solutions, and also
set
$$
\tilde{U}_{ih} = \Pi \tilde{U}_i
$$
for the update in the next time step.

\end{description}
\end{scheme}

\begin{rem}
If $L=0$, the LEQRK-DG-PC scheme reduces to Scheme \ref{LEQRKDG}.
\end{rem}

\section{Numerical results}
In this section, we numerically test the orders of convergence of the proposed LEQRK-DG-PC schemes. Further, we apply the schemes to the 2D Swift-Hohenberg equation in order to recover some known patterns, while
we also verify the unconditional energy stability at the same time.

The experimental orders of convergence (EOC) at $T=n\tau$ in terms of $h$ and $\tau$ are determined respectively by
$$
\text{EOC}=\log_2 \left( \frac{e_h^n}{e_{h/2}^n}\right),  \quad \text{EOC}=\log_2 \left( \frac{e_h^n}{e_{h}^{2n}}\right),
$$
where $e_h^n$ represents the error between the numerical solution $u_h^n(x, y)$ and the exact solution $u(x,y,t^n)$, and $e^{2n}_h$
 corresponds to the numerical solution with $\tau/2$ as the time step.

The Swift-Hohenberg equation is a special case of model equation (\ref{fourthPDE}) with $a=1$  and
\bq\label{shphi}
\Phi(u)=- \frac{\epsilon}{2}u^2 -\frac{g}{3}u^3 +\frac{u^4}{4},
\eq
that is,
\begin{align}\label{sh}
u_t  = -\Delta^2 u -2 \Delta u +(\epsilon-1)u +gu^2 -u^3.
\end{align}
Here physical parameters are $g\geq 0$ and $\epsilon \in \mathbb{R}$, which together with the size of the domain play an important role in pattern selection;  see,  e.g.,  \cite{BPT01, PR04, MD14}.
In our numerical tests, we focus on (\ref{shphi}) with $g \geq 0$ and $\epsilon>0$. This function has double wells with two local minimal values at
$u_\pm=  \frac{g\pm\sqrt{g^2+4\epsilon}}{2}$ such that $\Phi'(u_\pm)=0$, and
\bqs
\Phi(u) \geq \min \{\Phi(u_\pm)\}=
\min_{v=u_\pm} \left( -\frac{1}{12} \left(gv(g^2+4\epsilon)+\epsilon(g^2+3\epsilon) \right) \right)=-b,
\eqs
so it suffices to choose the method parameter $C_0\geq b$. In all numerical examples $b<1$, together with the discussion in Remark \ref{Bchoice}, we will take $C_0=10^3$ for all examples.

\begin{example}\label{Ex1dAccS} (Spatial accuracy test)
Consider the Swift-Hohenberg equation (\ref{sh}) with an added source of form
$$
f(x,y, t)=- \varepsilon v -gv^2+ v^3, \quad v: =e^{-t/4}\sin(x/2)\sin(y/2),
$$
subject to initial data
\bq\label{initex1}
u_0(x,y) = \sin(x/2)\sin(y/2).
\eq
This problem has an explicit solution
\bq\label{uexact}
u(x,y,t) = e^{-t/4}\sin(x/2)\sin(y/2).
\eq
To be specific, we take $\varepsilon=0.025$, $g=0$, and domain $\Omega=[-2\pi, 2\pi]^2$ with periodic boundary conditions. We shall test the LEQRK-DG-PC scheme based on the RK method with Butcher tableau (\ref{RK4th}) and $P^k$ polynomials. Note that due to the source term, we need to add
$$
(f(\cdot, t^{n}+b_i\tau), \phi),
$$
to the right hand side of both (\ref{RKIEQDG0}a) and (\ref{PRKIEQDG0}a).
In prediction step, we take $L=10$ and tolerance $Tol=10^{-10}$.
This example is used to test the spatial accuracy, using polynomials of degree $k$ with $k =1,\ 2,\ 3$ on 2D rectangular meshes.
Both errors and orders of convergence at $T=0.01$ are reported in Table \ref{tab2dacc}. These results confirm the $(k+1)$th orders of accuracy in $L^2, L^\infty$ norms.

\begin{table}[!htbp]\tabcolsep0.03in
\caption{$L^2, L^\infty$ errors and EOC at $T = 0.01$ with mesh $N\times N$.}
\begin{tabular}[c]{||c|c|c|c|c|c|c|c|c|c||}
\hline
\multirow{2}{*}{$k$} & \multirow{2}{*}{$\tau$}&   \multirow{2}{*}{ } & N=8 & \multicolumn{2}{|c|}{N=16} & \multicolumn{2}{|c|}{N=32} & \multicolumn{2}{|c||}{N=64}  \\
\cline{4-10}
& & & error & error & order & error & order & error & order\\
\hline
\multirow{2}{*}{1}  & \multirow{2}{*}{1e-3} & $\|u-u_h\|_{L^2}$ &  3.73985e-01 & 9.73764e-02 & 1.94 & 2.39651e-02 & 2.02 & 5.95959e-03 & 2.01  \\
\cline{3-10}
 & & $\|u-u_h\|_{L^\infty}$  & 1.38441e-01 & 3.83905e-02 & 1.85 & 9.61382e-03 & 2.00 & 2.40153e-03 & 2.00  \\
\hline
\hline
\multirow{2}{*}{2}  & \multirow{2}{*}{1e-4} & $\|u-u_h\|_{L^2}$ & 7.10034e-02 & 1.50739e-02 & 2.24 & 2.02727e-03 & 2.89 & 2.58614e-04 & 2.97  \\
\cline{3-10}
 & & $\|u-u_h\|_{L^\infty}$  & 2.41033e-02 & 3.22536e-03 & 2.90 & 4.40302e-04 & 2.87 & 5.63426e-05 & 2.97  \\
 \hline
\hline
\multirow{2}{*}{3} & \multirow{2}{*}{2e-5}  & $\|u-u_h\|_{L^2}$ & 1.20130e-02 & 1.13186e-03 & 3.41 & 7.72408e-05 & 3.87 & 4.94306e-06 & 3.97  \\
\cline{3-10}
 & & $\|u-u_h\|_{L^\infty}$   & 3.85682e-03 & 3.68735e-04 & 3.39 & 2.43500e-05 & 3.92 & 1.53904e-06 & 3.98  \\
\hline
\end{tabular}\label{tab2dacc}
\end{table}

\end{example}

\begin{example}\label{Ex1dAccT} (Temporal accuracy test)
Consider the Swift-Hohenberg equation (\ref{sh}) with an added source of form
$$
f(x,y, t)=- \varepsilon v -gv^2+ v^3, \quad v: =e^{-49t/64}\sin(x/4)\sin(y/4),
$$
subject to initial data
\bq\label{initex12}
u_0(x,y) = \sin(x/4)\sin(y/4).
\eq
Its exact solution is given by
$$
u(x,y,t) = e^{-49t/64}\sin(x/4)\sin(y/4).
$$
We want to test the temporal accuracy of the LEQRK-DG-PC method,
for which we take $\varepsilon=0.025, g=0$, and domain $\Omega= [-4\pi, 4\pi]^2$ with periodic boundary conditions.
We apply the two-stage LEQRK-DG-PC schemes based on second, third and fourth order RK methods with Butcher tableau (\ref{RK2nd})-(\ref{RK4th}) and $P^3$ polynomials.
Similar to Example \ref{Ex1dAccS}, we  also need to add
$$
(f(\cdot, t^{n}+b_i\tau), \phi),
$$
to the right hand side of both (\ref{RKIEQDG0}a) and (\ref{PRKIEQDG0}a).
We take time steps $\tau=2^{-m}$ for $2\leq m\leq 5$ and mesh size $64\times 64$. In the prediction step, we choose the tolerance $Tol=10^{-10}$ and the value of $L$ depends on the specific RK methods. The $L^2, L^\infty$ errors and orders of convergence at $T=1.5$ are shown in Table \ref{timeacc}, and these results confirm that the schemes as tested can achieve the optimal orders of convergence in time.

\begin{table}[!htbp]\tabcolsep0.03in
\caption{$L^2, L^\infty$ errors and EOC at $T = 1.5$ with time step $\tau$.}
\begin{tabular}[c]{||c|c|c|c|c|c|c|c|c|c||}
\hline
\multirow{2}{*}{RK} & \multirow{2}{*}{~$L$~}&   \multirow{2}{*}{ } & $\tau=2^{-2}$ & \multicolumn{2}{|c|}{$\tau=2^{-3}$} & \multicolumn{2}{|c|}{$\tau=2^{-4}$} & \multicolumn{2}{|c||}{$\tau=2^{-5}$}  \\
\cline{4-10}
& & & error & error & order & error & order & error & order\\
\hline
\multirow{2}{*}{(\ref{RK2nd})}  & \multirow{2}{*}{$0$} & $\|u-u_h\|_{L^2}$ &  5.84575e-02 & 1.29975e-02 & 2.17 & 3.21365e-03 & 2.02 & 8.05270e-04 & 2.00  \\
\cline{3-10}
 & & $\|u-u_h\|_{L^\infty}$  & 5.51717e-03 & 1.13568e-03 & 2.28 & 2.80093e-04 & 2.02 & 7.04859e-05 & 1.99  \\
\hline
\hline
\multirow{2}{*}{(\ref{RK3rd})}  & \multirow{2}{*}{$2$} & $\|u-u_h\|_{L^2}$ &  6.49591e-03 & 7.15397e-04 & 3.18 & 8.88739e-05 & 3.01 & 9.69107e-06 & 3.20  \\
\cline{3-10}
 & & $\|u-u_h\|_{L^\infty}$  & 7.59053e-04 & 1.09547e-04 & 2.79 & 1.37766e-05 & 2.99 & 1.46945e-06 & 3.23  \\
\hline
\hline
\multirow{2}{*}{(\ref{RK4th})}  & \multirow{2}{*}{$2$} & $\|u-u_h\|_{L^2}$ & 2.10020e-03 & 1.38306e-04 & 3.92 & 7.30941e-06 & 4.24  & $--$ & $--$ \\
\cline{3-10}
 & & $\|u-u_h\|_{L^\infty}$  & 3.43273e-04 & 2.20772e-05 & 3.96 & 1.42833e-06 & 3.95 & $--$ & $--$ \\
\hline
\end{tabular}\label{timeacc}
\end{table}

\end{example}

\begin{example}\label{Ex2dPatt2} (Rolls and Hexagons)
In this example, we simulate the formation and evolution of patterns of the the Swift-Hohenberg equation (\ref{sh}), which arises in the Rayleigh-B\'{e}nard convection. Following \cite{PCC14, LY19}, we run the simulation from $t=0$ to $t=198$ on a rectangular domain $\Omega=[0,100]^2$, subject to random initial data and periodic boundary conditions. Model parameters $\varepsilon, \ g$  will be specified below for different cases.

We apply the LEQRK-DG-PC scheme based on the fourth order RK method with Butcher tableau (\ref{RK4th}) and $P^2$ polynomials using mesh $128 \times 128$. We take time step $\tau =0.1$, which is much larger than that used in \cite{PCC14, LY19}. In the following two test cases, we output $E(q_h^n, U_h^n)-C_0|\Omega|$ instead of $E(q_h^n, U_h^n)$  to  better observe the evolution of the original free energy $\mathcal{E}(u)$.

\noindent\textbf{Test case 1.} (Rolls)
For parameters  $\varepsilon=0.3, \ g=0$, we observe the periodic rolls for different times as shown in Figure \ref{PatBifur2}. We see that the pattern evolves approaching the steady-state after $t>60$, as also evidenced by the energy evolution plot in Figure \ref{BifEng3}a.

 \begin{figure}
 \centering
 \subfigure{\includegraphics[width=0.325\textwidth]{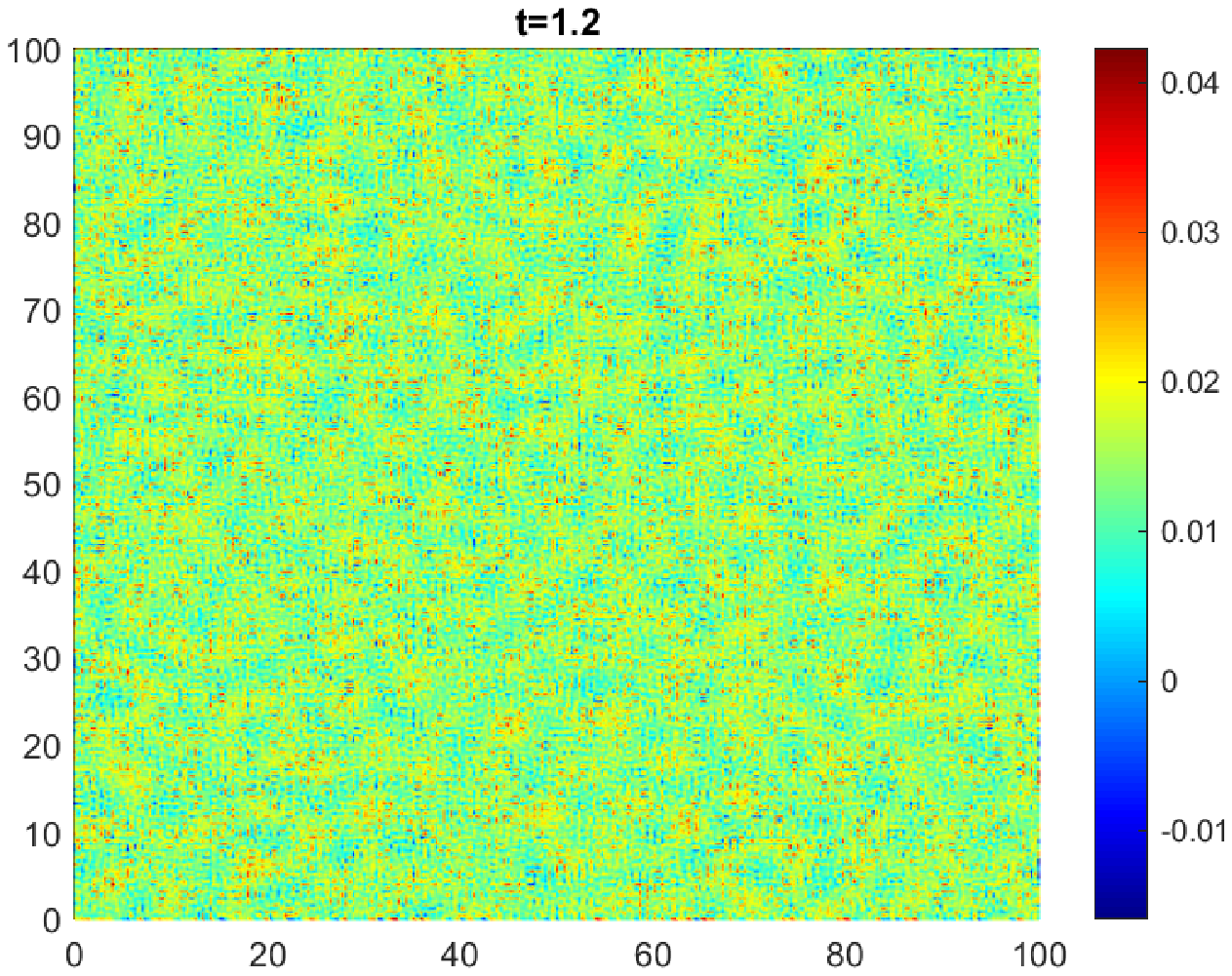}}
 \subfigure{\includegraphics[width=0.325\textwidth]{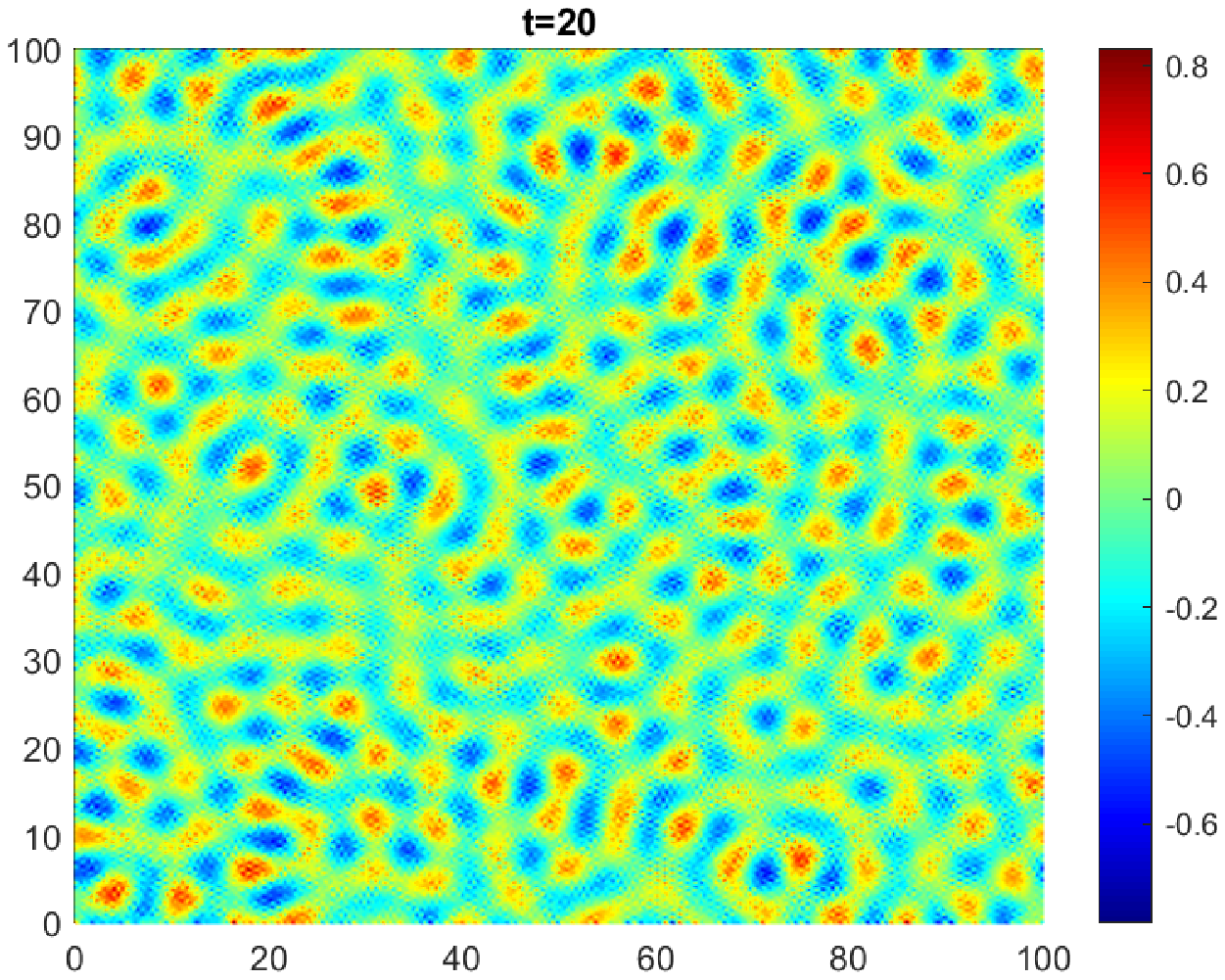}}
 \subfigure{\includegraphics[width=0.325\textwidth]{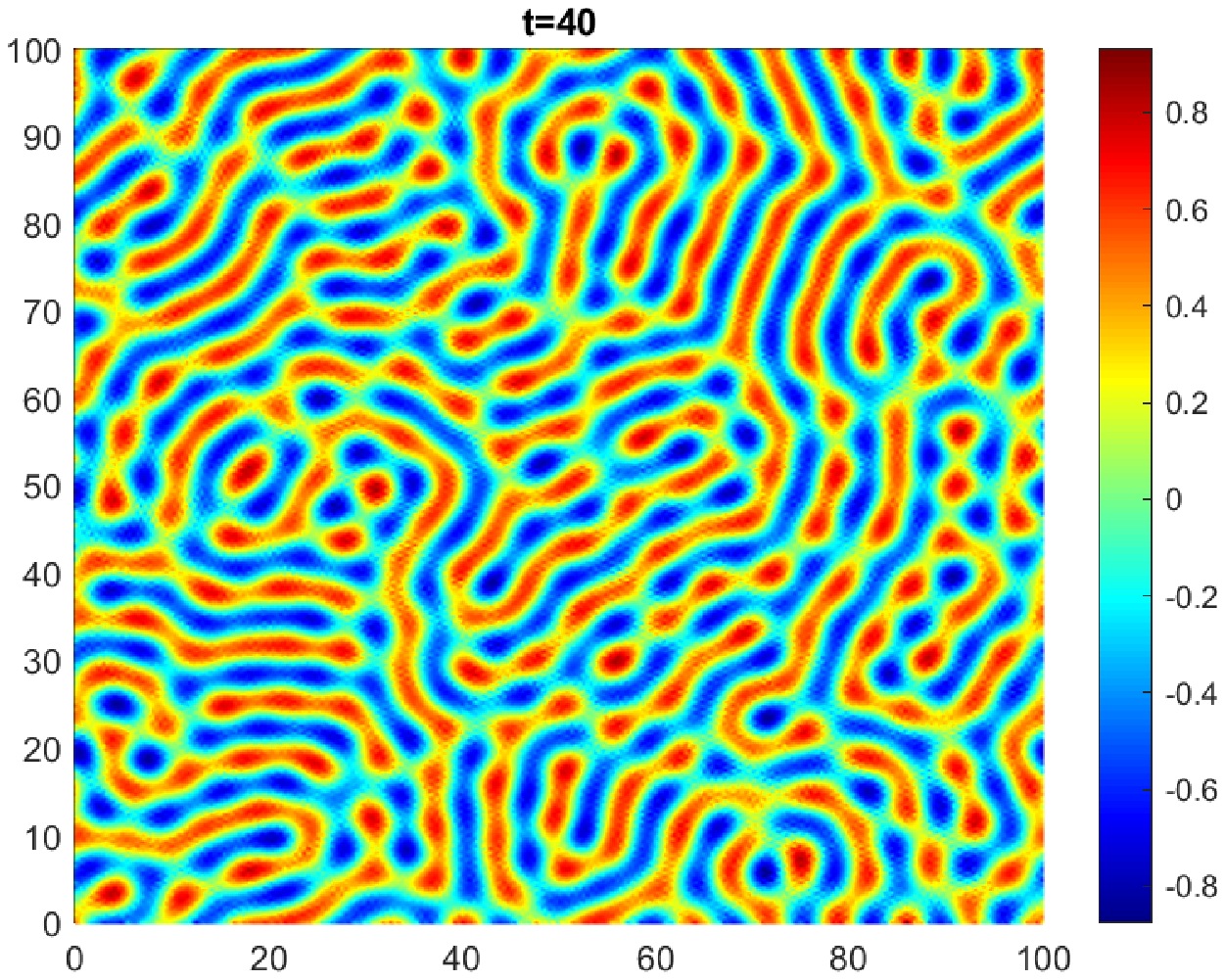}}
 \subfigure{\includegraphics[width=0.325\textwidth]{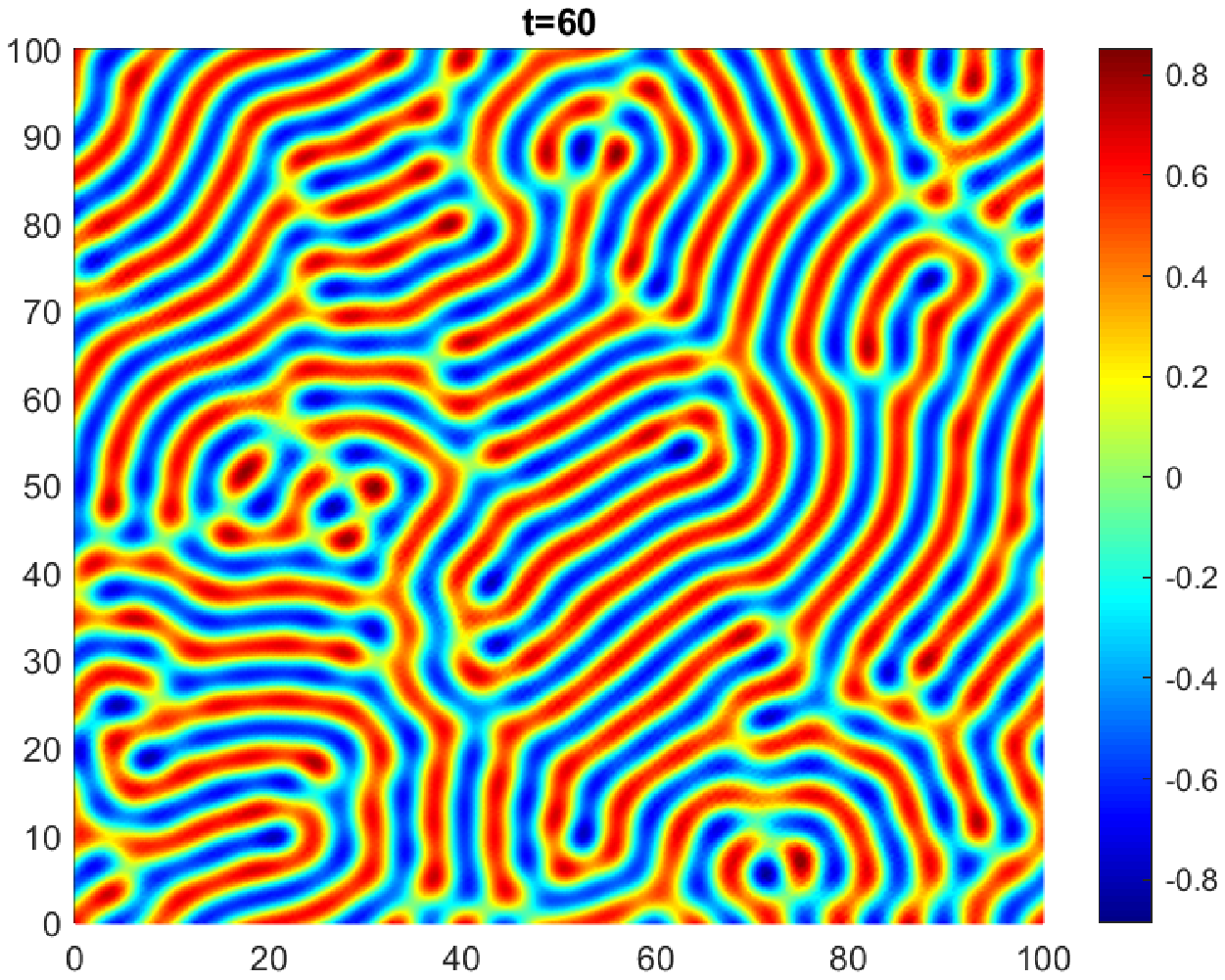}}
 \subfigure{\includegraphics[width=0.325\textwidth]{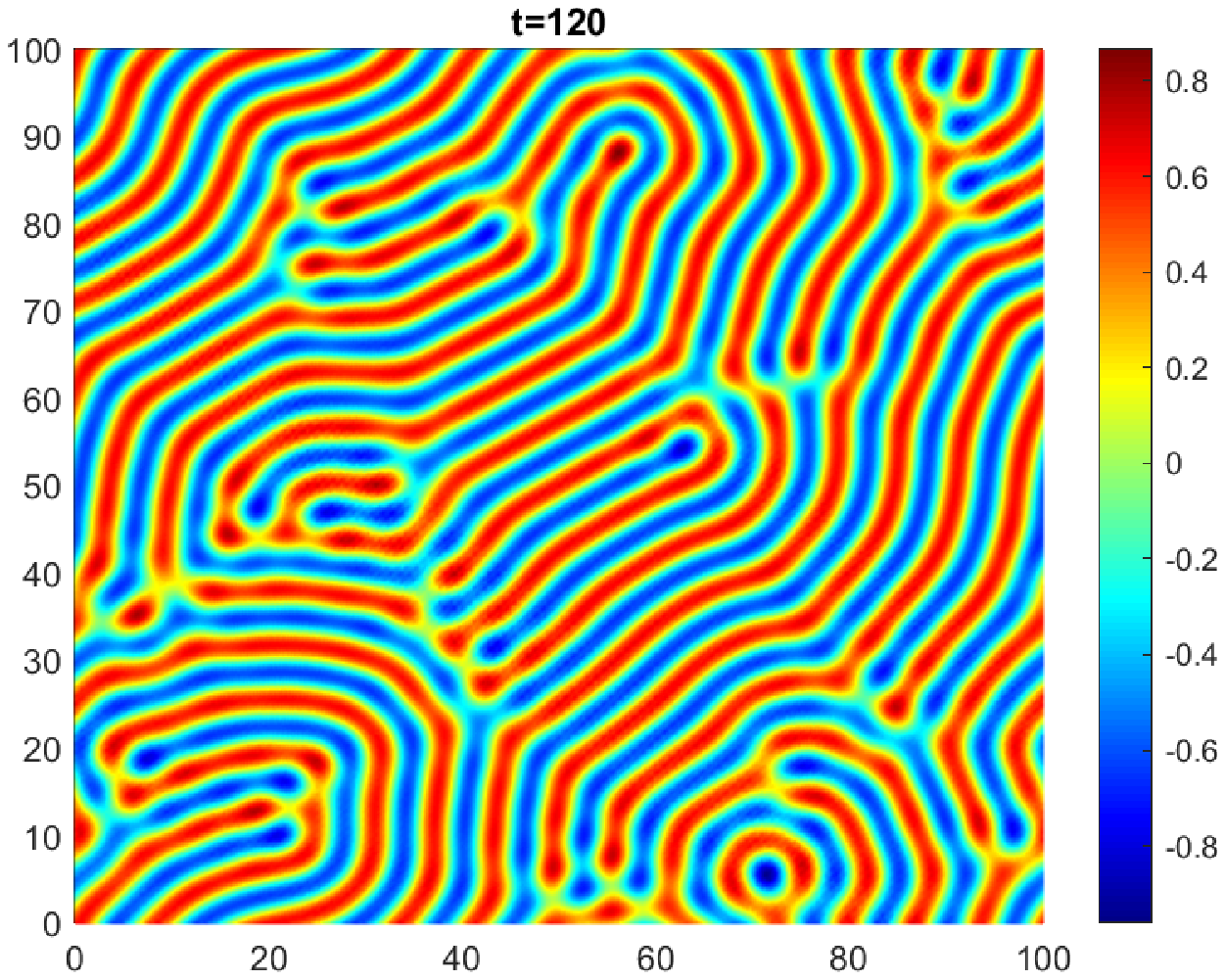}}
 \subfigure{\includegraphics[width=0.325\textwidth]{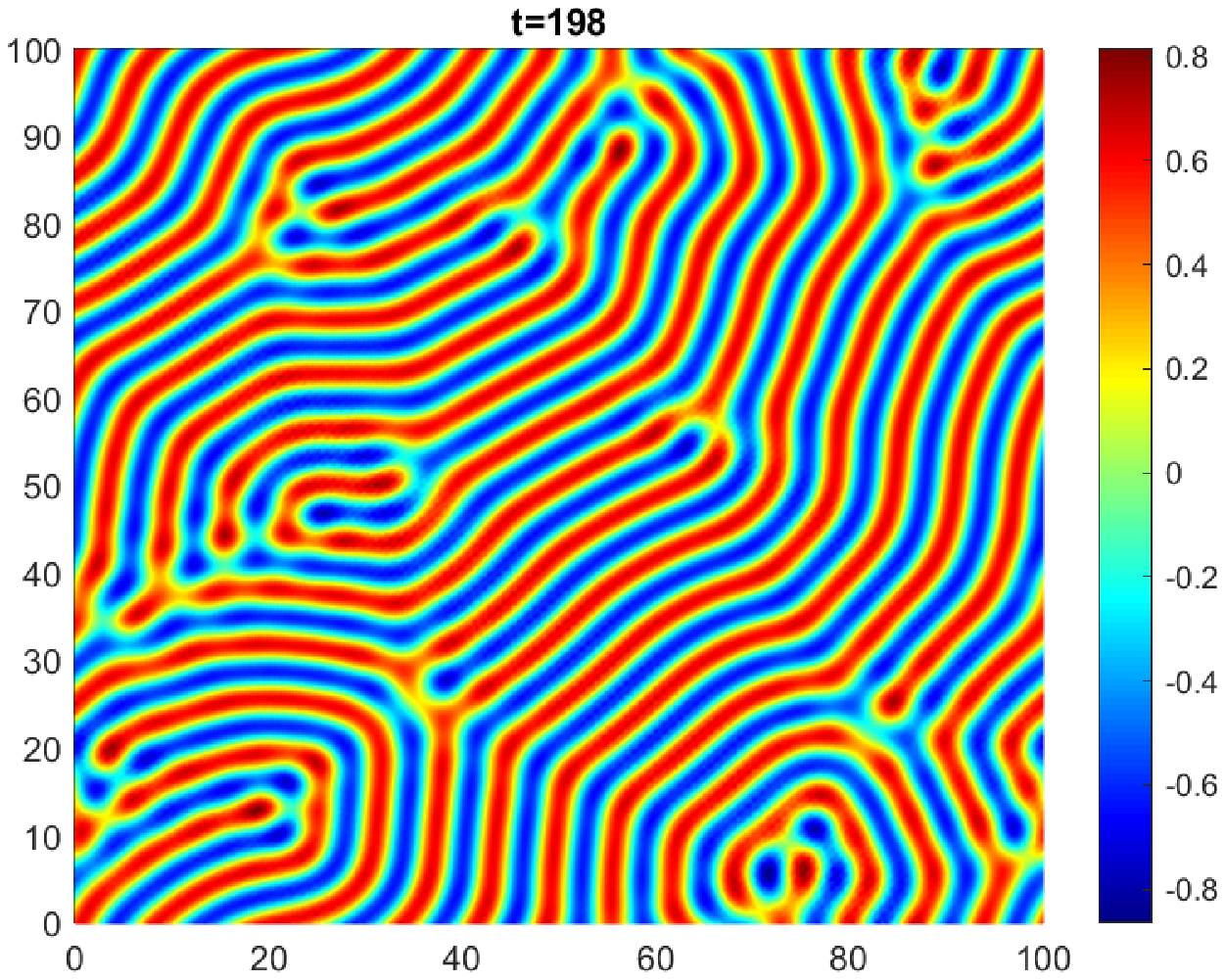}}
  \caption{ Evolution of periodic rolls.
  } \label{PatBifur2}
 \end{figure}

\noindent\textbf{Test case 2.} (Hexagons) The numerical simulations with the parameters $\varepsilon=0.1, \ g=1.0$ reveal vividly the formation and evolution of the hexagonal pattern as shown in Figure \ref{PatBifur3}. The pattern at $t=1.2$ is similar to that of rolls as shown in Figure \ref{PatBifur2}. Similar to the pattern obtained by the IEQ-DG scheme in \cite{LY19}, we also observe that at a certain point before $t=40$, lines break up giving way to single droplets that take hexagonal symmetry. The steady state is approaching after $t>100$.

The evolution of the patterns for both cases is shown to satisfy the energy dissipation law in Figure \ref{BifEng3}. With the same parameters $\varepsilon, \ g$ as in \cite{LY19}, the LEQRK-DG-PC scheme can generate  quite similar formation and evolution of both roll and hexagonal patterns even with a larger time step.

\begin{figure}
\centering
\subfigure{\includegraphics[width=0.325\textwidth]{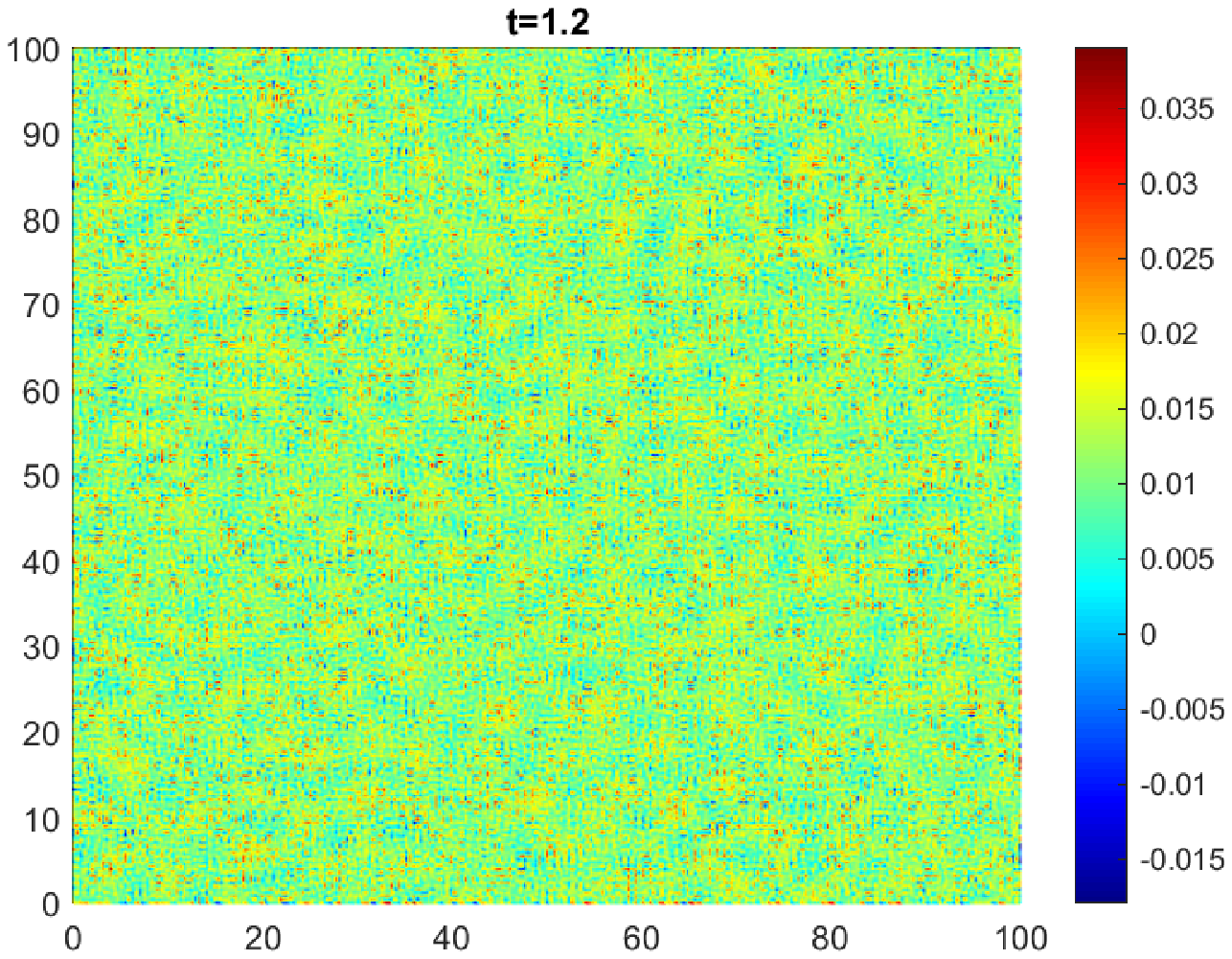}}
\subfigure{\includegraphics[width=0.325\textwidth]{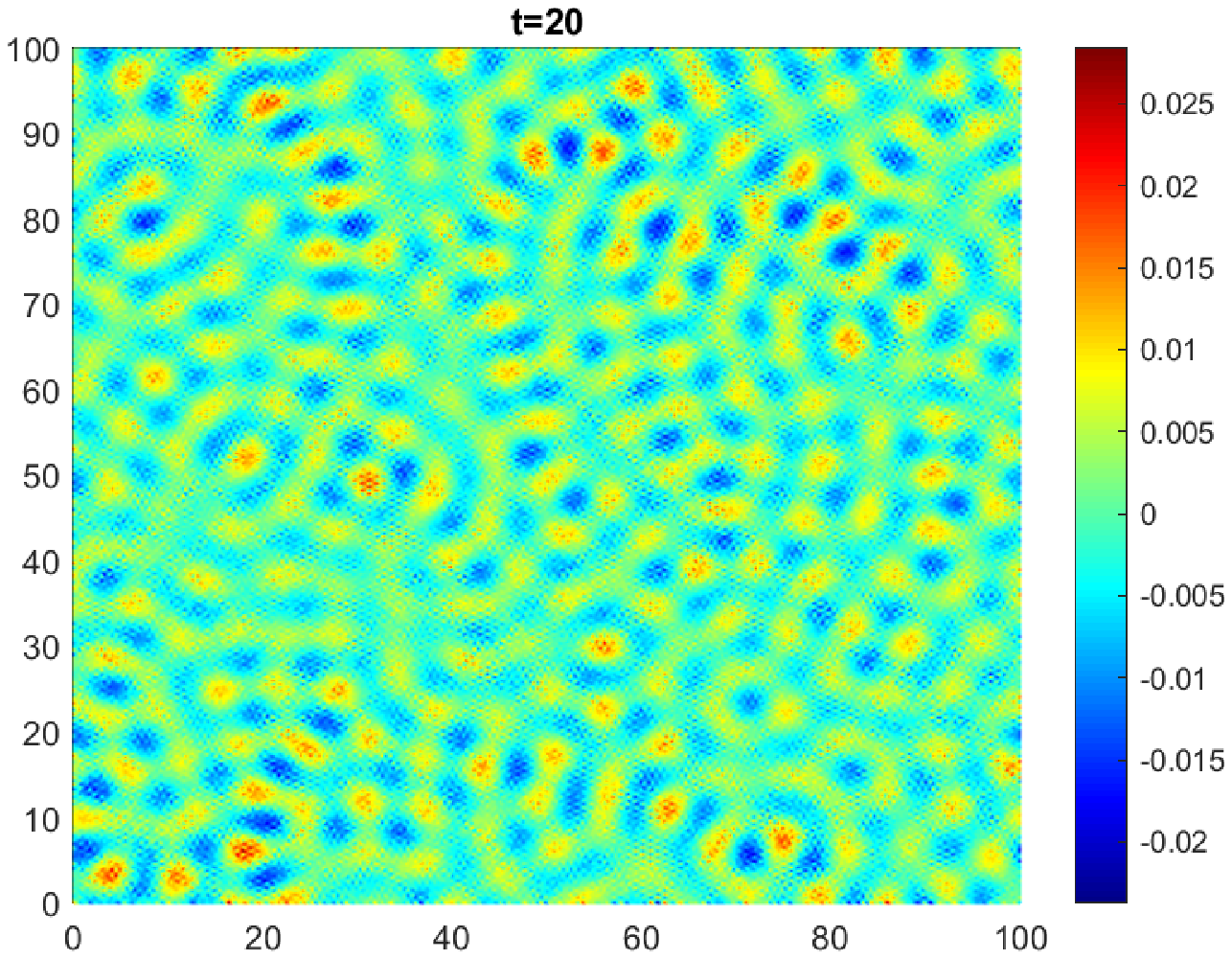}}
\subfigure{\includegraphics[width=0.325\textwidth]{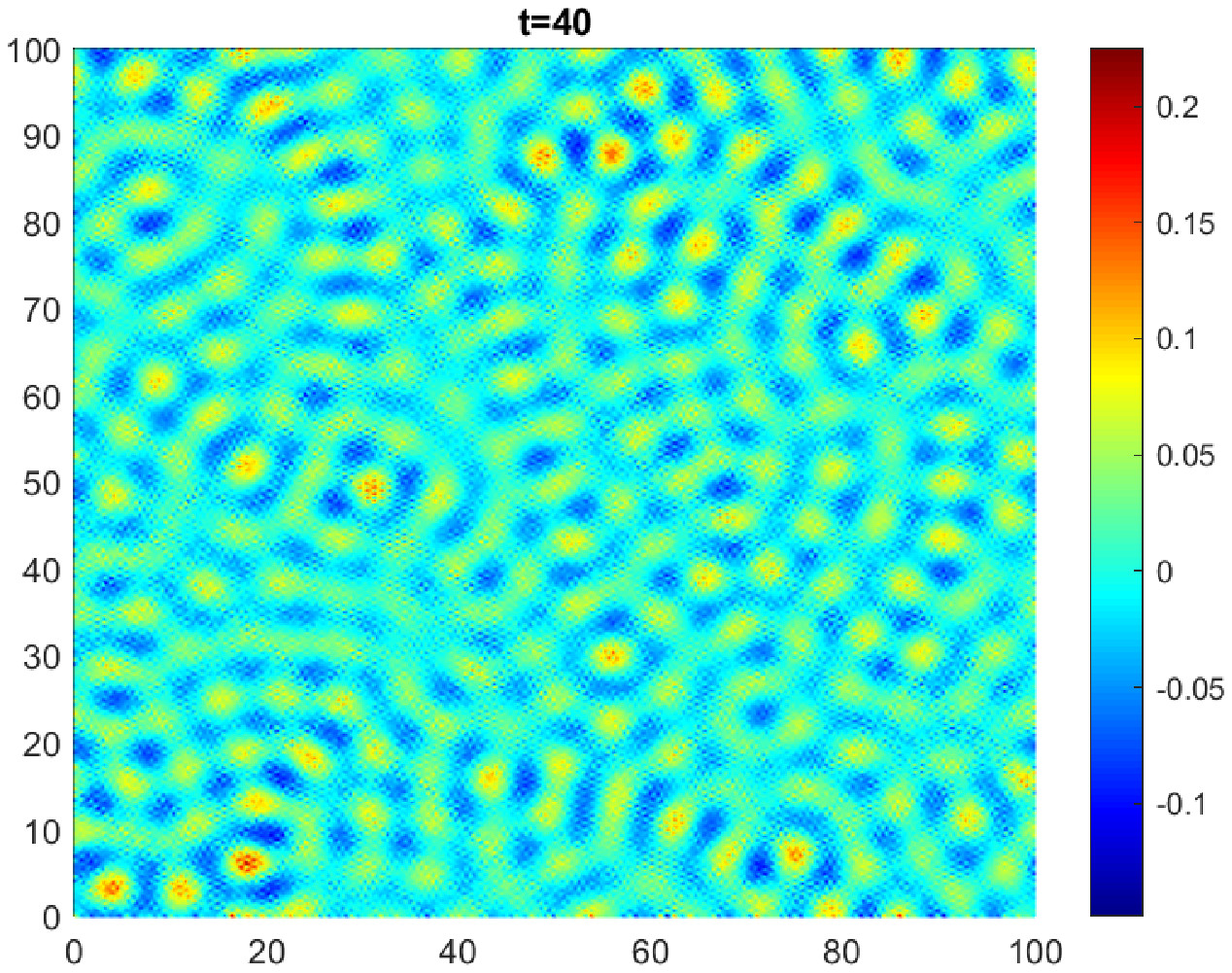}}
\subfigure{\includegraphics[width=0.325\textwidth]{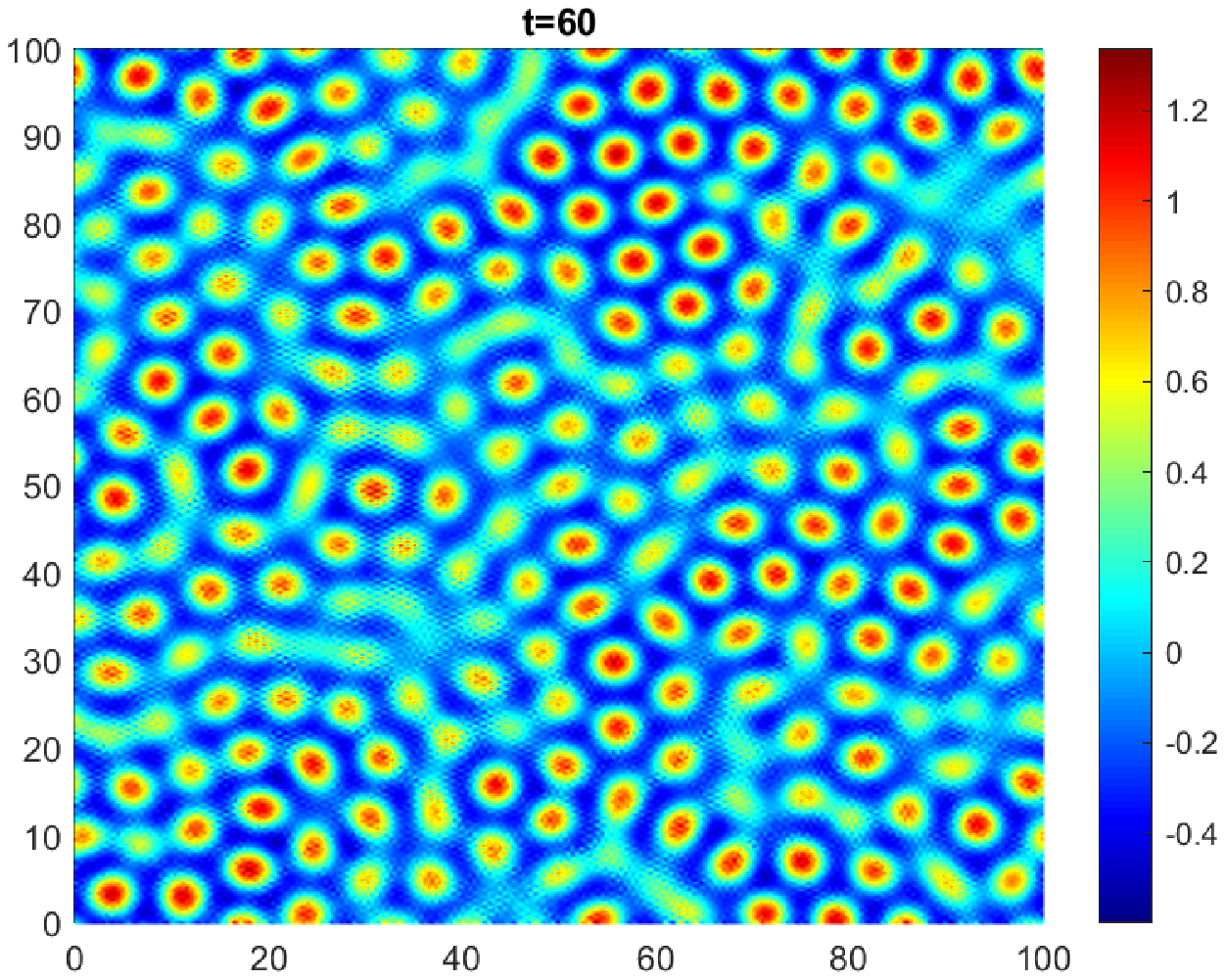}}
\subfigure{\includegraphics[width=0.325\textwidth]{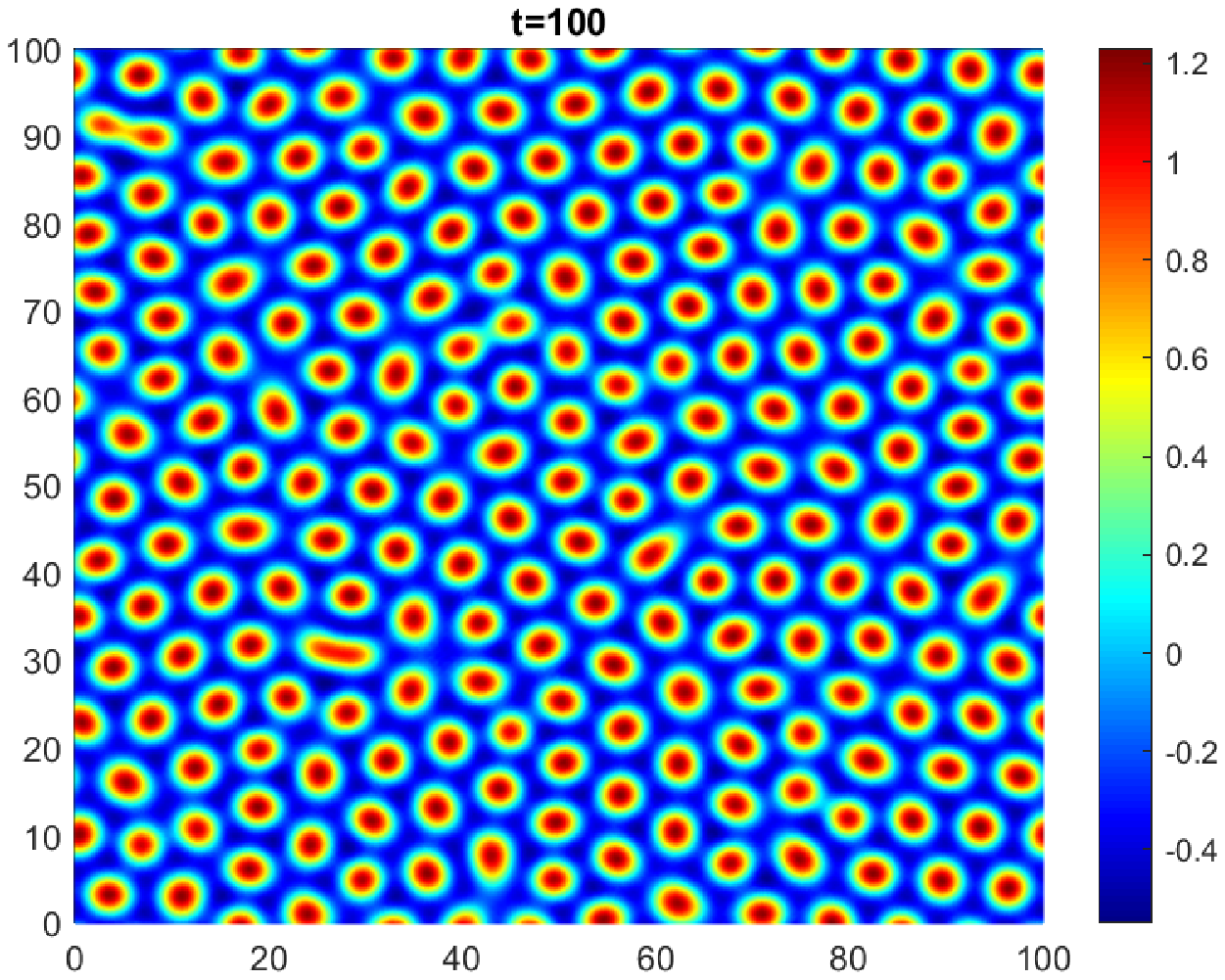}}
\subfigure{\includegraphics[width=0.325\textwidth]{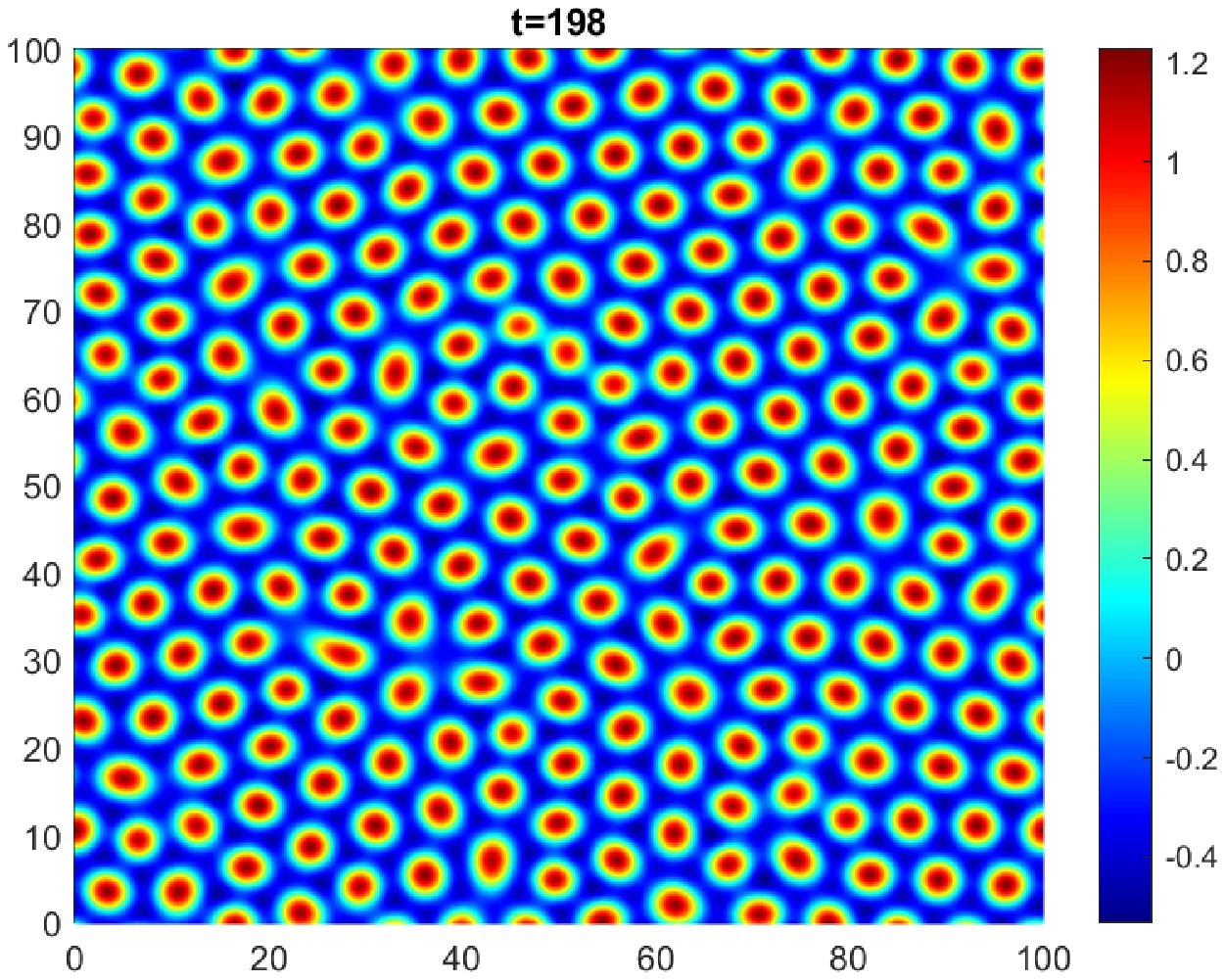}}
  \caption{ Evolution of hexagonal patterns. 
  } \label{PatBifur3}
 \end{figure}

 \begin{figure}
 \centering
 \subfigure[]{\includegraphics[width=0.49\textwidth]{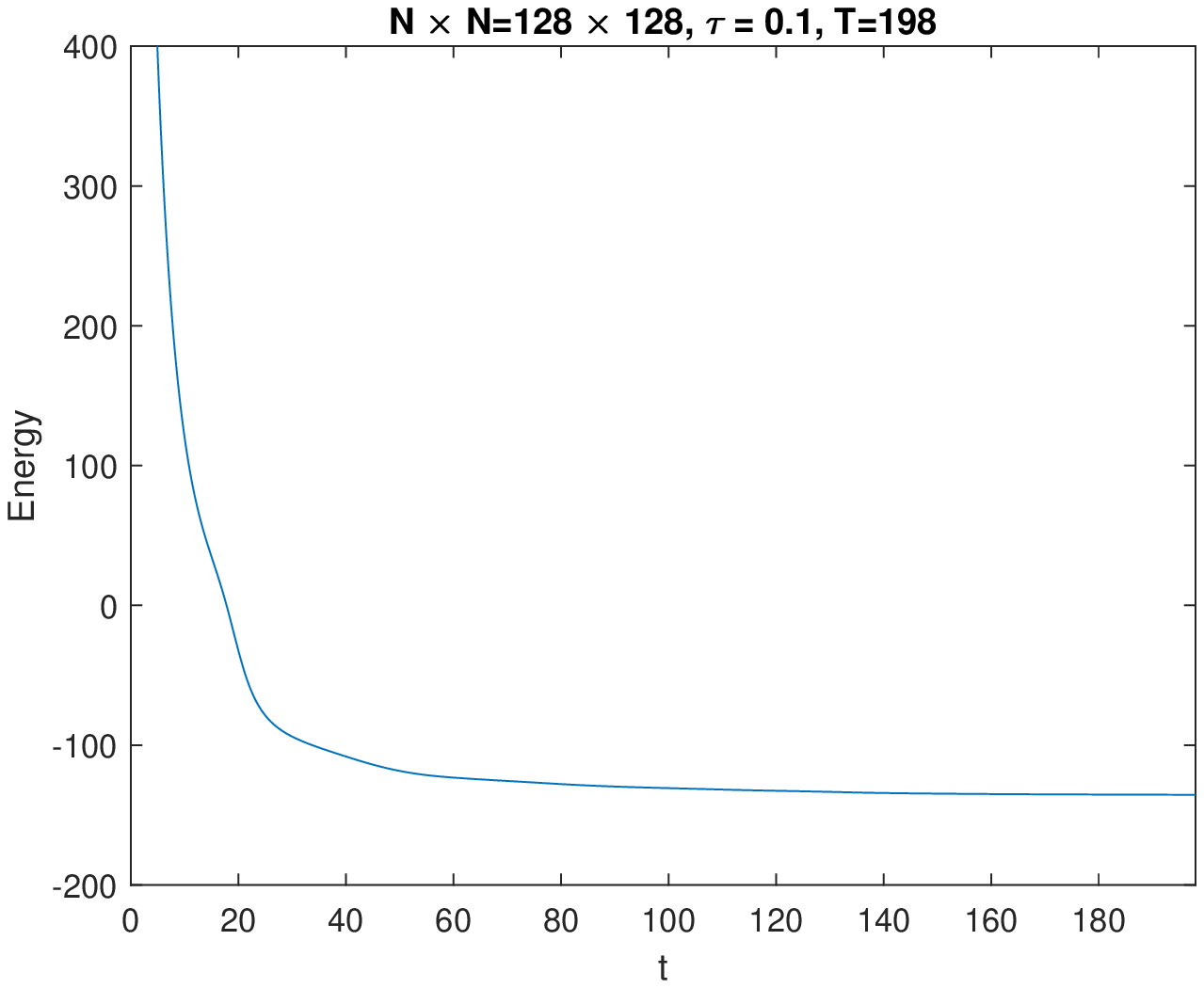}}
 \subfigure[]{\includegraphics[width=0.49\textwidth]{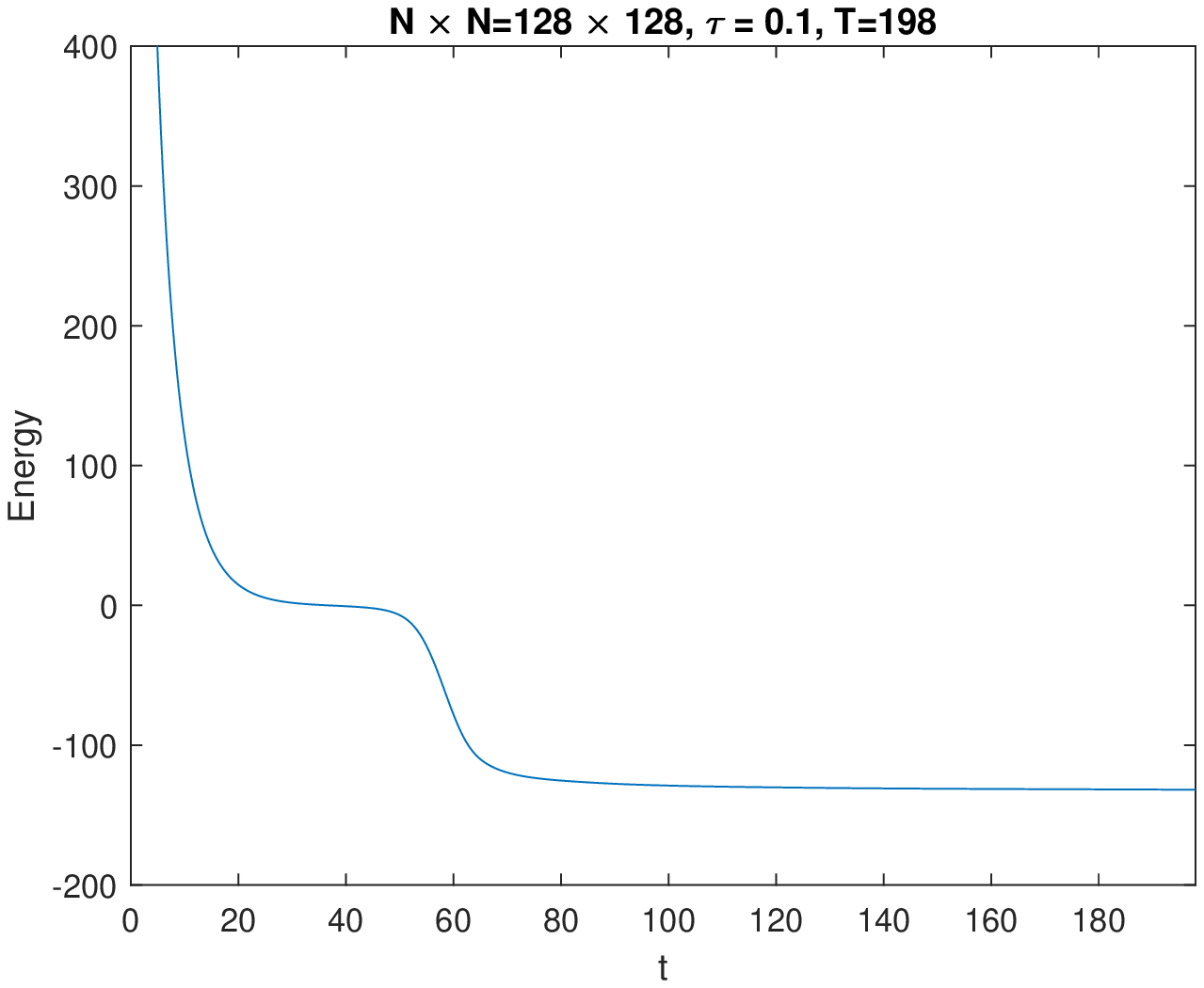}}
 \caption{ Energy evolution. (a) Rolls. (b) Hexagons.
 } \label{BifEng3}
 \end{figure}

\end{example}

\section{Concluding remarks}
In this paper, we present a new class of arbitrarily high order, fully discrete DG schemes. These schemes have several advantageous properties: (1) the schemes are all linear such that they are easy to implement and computationally efficient; (2) the schemes are uniquely solvable and unconditionally energy stable, these ensure that large time steps can be used in some long time simulations; (3) the schemes can reach arbitrarily high order of accuracy in both space and time, so that desired accuracy of solutions can be guaranteed with flexible meshes and time steps;  (4) the schemes do not depend on the specific form of the DG operator explicitly such that it can be applied to a larger class of DG schemes as long as they satisfy a semi-discrete energy dissipation law.
The proofs for energy stability are given. Several numerical examples are presented to assess the scheme performance in terms of accuracy and energy stability. The numerical results on two dimensional pattern formation problems indicate that the method is able to deliver expected patterns of high accuracy with a larger time step on coarse meshes.

\section*{Acknowledgments}
 This research was supported by the National Science Foundation under Grant DMS1812666.

\end{document}